\documentclass[11pt, reqno]{amsart}
\usepackage{amsmath,amsthm,amssymb, enumerate, tikz, multicol}

\newtheorem{thm}{Theorem}[section]
\newtheorem{cor}[thm]{Corollary}
\newtheorem{lem}[thm]{Lemma}

\numberwithin{equation}{section}

\newcommand{\Z}{\mathbb{Z}}
\newcommand{\Q}{\mathbb{Q}}

\newcommand{\Chi}{\chi}

\begin{document}
\title{The class number formula for imaginary quadratic fields}

\author{Joseph Lewittes
}
\address{Department of Mathematics and Computer Science\\
Lehman College CUNY \\
}
\email{joseph.lewittes@lehman.cuny.edu}

\date{\today}

\begin{abstract}
It is shown that the class number for negative discriminant $D$ can be expressed in terms of the base $B$ expansions of reduced fractions $\frac{x}{|D|}$, where $B$ is an integer prime to $D$. This result is then formulated to obtain information about the distribution of the values of $\Chi(x)$, where $\Chi$ is the quadratic character associated to $D$. This leads to simplified formulas for the class number in certain cases. 
\end{abstract}

\subjclass[2010]{Primary 11R11; Secondary 11R29, 11L40, 11Y40.}

\keywords{class number formula, imaginary quadratic fields}

\maketitle

\section{Introduction}   
Associated to an imaginary quadratic number field $K$ are three important items: $D$, the discriminant; $h$, a positive integer which is the order of the ideal class group; $\Chi$, a quadratic character which governs how rational primes factor in $K$. The field $K$ is uniquely determined by its discriminant. To indicate the dependence of $h$, $\Chi$ on $D$, we write $h(D)$, $\Chi_{D}$, except in cases where $D$ is clear from the context and so $h$, $\Chi$ suffice. Below, $\Chi$ will be given explicitly. Dirichlet (writing in the framework of Gauss' theory of binary quadratic forms) proved a class number formula for $h$, which in modern form is 

\begin{equation} \label{eq1}
h(D)=-\frac{1}{|D|}{\sum_{x = 1}^{|D|}\Chi_{D}(x)x}
\end{equation}

Actually this is valid only for $D<-4$, which we assume throughout; for the excluded cases $D =-3,-4$ a minor correction is needed which does not concern us here. For our purposes, one need not know the actual significance of $D,h,\Chi$ for the field $K$. All of our effort will be concentrated on the sum on the right side of the formula, which involves only rational arithmetic. For further information, one may consult \cite{BoSha}, pages 234-238, 342-347. Here we present only some necessary definitions and notation. The paper \cite{Berndt} deals with character sums but the techniques and results there have little overlap with our methods and conclusions here.

Every $K$ is uniquely of the form $\Q(\sqrt{m})$, where $m$ is a negative square-free integer. $D$ is then defined to be $D = m$, if $m \equiv 1 \pmod{4}$ and $D=4m$ otherwise. We always set $N=|D|$. $\Chi$ is an odd Dirichlet quadratic character mod $N$. Concretely this means $\Chi:\Z \rightarrow \{0,1,-1\}$ with the following properties:
\begin{enumerate}
\item $\Chi(a)=0$ if gcd$(a, N)>1, \Chi(a)=1$ or $-1$ if gcd$(a, N)=1$ \label{eqi}
\item $\Chi(a)= \Chi(b)$ whenever $a \equiv b \pmod{N}$ \label{eqii}
\item $\Chi(ab)= \Chi(a)\Chi(b)$ \label{eqiii}
\item $\Chi(-1)=-1$. \label{eqiv}
\end{enumerate}

Note that in (\ref{eq1}), $\Chi(x)=0$ whenever gcd$(x, N)>1$, so such an $x$ makes no contribution to the sum. For our applications the non-zero values of $\Chi(x)$ need to be known explicitly. The simplest case is when $D=m \equiv 1 \pmod{4}$, in which case $\Chi_{D}(x)=\left(\frac{x}{|m|}\right)$, the Jacobi symbol. $D \equiv 0 \pmod{4}$ is somewhat more complicated. For this we introduce the character $\Chi_{4}(x)=(-1)^{\frac{x-1}{2}}$, whose values are $1,-1$ according as $x \equiv 1$ or $x \equiv 3 \pmod{4}$; also the character $\Chi_{8}(x)=(-1)^{\frac{x^2-1}{8}}=1$ or $-1$ according as $x \equiv 1,7 \pmod{8}$ or $x\equiv 3,5 \pmod{8}$. Then with $D=4m$,

$$
\Chi_D(x) = \left\{
\begin{array}{ll}
\Chi_{4}(x)\left(\frac{x}{|m|}\right); & \text{if} \; m \equiv 3 \pmod{4}\\
\Chi_{8}(x)\left(\frac{x}{|n|}\right); & \text{if} \; m=2n,n \equiv 1\pmod{4}\\ 
\Chi_{4}(x)\Chi_{8}(x)\left(\frac{x}{|n|}\right); & \text{if}\; m=2n,n \equiv 3 \pmod{4} 
\end{array}
\right.
$$

The motivation for this paper was an article by K. Girstmair, \cite{KG}. I want to thank Professor Pieter Moree who alerted me to \cite{KG}, pointing out its relevance to some previous work of mine. Girstmair's result is as follows. Let $p>3$ be a prime $\equiv 3 \pmod{4}$, $B$ a primitive root mod $p$ and let $\frac{1}{p} = {\sum_{i = 1}^{\infty}\frac{a_i}{B^{i}}}$ be the base $B$ expansion of the fraction $\frac{1}{p}$. Then

\begin{equation} \label{eq2}
(B+1)h(-p)={\sum_{i = 1}^{p-1}(-1)^{i}a_i}.
\end{equation}

Here $-p \equiv 1\pmod{4}$ is the discriminant of the field $K=\Q{(\sqrt{-p})}$ and the related character is $\left(\frac{x}{p}\right)$, the Legendre symbol. This is certainly an interesting result, but it is limited to the special case $D=-p$, with $B$ a primitive root mod $p$. In the next section it will be shown that an analogous formula holds for any $D$ with any base $B$ prime to $D$. Section 3 then shows how the base $B$ formula can be recast in terms of $\Chi$ to produce simpler class number formulas, which give information about the distribution of the values of $\Chi(x)$ in certain intervals. Then in Sections 4 and 5, applications of the new formulas to the cases $D \equiv 1\pmod{4}$ and $D \equiv 0 \pmod{4}$, respectively, are presented. A sample of one such result is Corollary \ref{cor43}:

\begin{equation} \label{eq3}
\mbox{if} \; D \equiv 1 \pmod{4} \; \mbox{and} \; 3\not| D, \; \mbox{then}\; h(D)=\left|{\sum_{1 \leq x < \frac{N}{6}}\Chi(x)}\right|.
\end{equation}

For a simple numerical example of (\ref{eq1}) and (\ref{eq2}), take $D=-7$. By (\ref{eq1}), $h(-7)=-\frac{1}{7}{\sum_{x=1}^{6}\left(\frac{x}{7}\right)x}$. (Note that it is not a priori obvious that the right side is an integer or positive, though by definition $h$ is always a positive integer. This is part of the magic of the class number formula.) Evaluating the sum gives $h(-7) = -\frac{1}{7}\left((1)1+(1)2+(-1)3+(1)4+(-1)5+(-1)6\right)=1$. (Observe by (\ref{eq3}), $h(-7)=|\Chi(1)|=1$, one step). Now let $B=10$, a primitive root mod $7$, then the base 10 expansion of $\frac{1}{7}$ is the well-known decimal $0.\overline{142857}$, the bar indicating endless repetition of the period block 142857. Now consider (\ref{eq2}). The left side is $(10+1)h(-7)=11$ and the right side is $-1+4-2+8-5+7=11$, which illustrates Girstmair's proposition. 

When doing numerical examples it is useful to have a table of values of $h(D)$. One such table is in \cite{BoSha}, Table 4, p. 425-426. This table gives $h(a)$ where $a=|m|$ in our notation; so to find $h(D)$ look for $h(a)$ with $a=|D|$ if $D \equiv 1\pmod{4}$, and $a=|D|/4$ if $D \equiv 0\pmod{4}$. (Note that the continuation of Table 4 to p. 426 has an incorrect heading).

\section{Base B expansions}

Let $N$ be an integer $>1$ and $X=\{x : 1 \leq x \leq N$ and gcd$(x,N)=1\}$. Denoting by $|S|$ the number of elements in the finite set $S, |X|=\phi(N),\phi$ being Euler's function. We shall often make use of the obvious fact that if $x,x' \in X$ and $x'\equiv x \pmod{N}$ then $x'=x$. From now on $x$ always denotes an element of $X$. For an integer $B>1$ the numbers $0,1,...,B-1$ are called the $B$-digits; there are $B$ of them. Expanding a real number in base $B$ is a well-known procedure; here we only discuss what is needed for our purposes. We assume always that $B$ is relatively prime to $N$. The base $B$ expansion of a fraction $\frac{x}{N}$ means an infinite series $\sum_{i = 1}^{\infty}\frac{a_i}{B^{i}}$ where each $a_i$ is a $B$-digit and the series converges to $\frac{x}{N}$. Such a series is found by the elementary school long division of $x$ by $N$, which we call LDA, the long division algorithm. It amounts to the following. Set $x_{1}=x$ and use integer division to divide $Bx_{1}$ by $N$, producing the quotient $a_1$ and remainder $x_2: Bx_{1}=a_{1}N+x_{2}, 0 \leq x_{2} < N$. $Bx_{1}>0$ implies $a_{1} \geq 0$ and as $B$, $x_{1}$ are both relatively prime to $N$, $Bx_1$ is also, hence $\frac{Bx_1}{N}$ is not an integer so $x_{2}>0$. Noting $x_{2} \equiv Bx_{1} \pmod{N}$, one sees $x_{2}$ is prime to $N$, so $x_{2} \in X$. $\frac{Bx_1}{N} = a_{1}+\frac{x_{2}}{N}$ shows $a_{1} < \frac{Bx_{1}}{N} < a_{1}+1$, so $a_{1}=\left[\frac{Bx_{1}}{N}\right]$, where, as usual, $[t]$ denotes the greatest integer $\leq t$. Finally $\frac{x_{1}}{N} < 1$ shows $\frac{Bx_{1}}{N} < B$, so $0 \leq a_{1} \leq B-1$ and $a_{1}$ is a $B-$digit. Now this process may be iterated to produce an infinite sequence of equations 

\begin{equation}
\begin{array}{l}
Bx_{1}=a_{1}N+x_{2} \\
Bx_{2}=a_{2}N+x_{3} \\
\vdots \\
Bx_{i-1}=a_{i-1}N+x_{i} \\
Bx_{i}=a_{i}N+x_{i+1} \\
\vdots \\
\end{array} 
\label{eq4}
\end{equation}

Each $a_{i}$ is a $B-$digit, each $x_{i} \in X$, $a_{i}=\left[\frac{Bx_{i}}{N}\right]$. An easy inductive argument shows that for $i \geq 1$, $\frac{x_1}{N}=\frac{a_1}{B^1}+\frac{a_2}{B^2}+...+\frac{a_i}{B^i}+\frac{x_{i+1}}{B^{i}N},~ 0 < \frac{x_{i+1}}{B^{i}N} < \frac{1}{B^{i}} \rightarrow 0$ as $i \rightarrow \infty$, so $\sum_{x = 1}^{\infty}\frac{a_i}{B^i}$ converges to $\frac{x_1}{N}$, providing the base $B$ expansion for $\frac{x}{N}$. Working backwards from equation $i$ we have $x_{i+1} \equiv Bx_{i} \equiv B^{2}x_{i-1} \equiv ... \equiv B^{i}x_{1}\pmod{N}.$ Let $e$ be the order of $B$ mod $N$, the smallest positive integer such that $B^{e} \equiv 1\pmod{N}$; by Euler's theorem $e|\phi(N)$. The $e$ numbers $x_{1},x_{2},...,x_{e}$ are all distinct, because $x_{i}=x_{j}$ for $1 \leq i < j \leq e$ implies $B^{j-1}x_{1} \equiv x_{j} = x_{i} \equiv B^{i-1}x_{1}\pmod{N}$, hence $B^{j-i} \equiv 1\mod{N}$, contradicting the definition of $e$. On the other hand, $x_{e+1} \equiv B^{e}x_{1} \equiv x_{1}(mod~N)$ implies $x_{e+1} = x_{1}$. Thus in (\ref{eq4}) equation $e+1$ must coincide with equation $1$ and in general equation $e+i$ coincides with equation $i$, for all $i$. Thus the LDA consists of the first $e$ equations and then the block repeats forever. In particular the digits in the $B$ expansion are periodic with period $e: a_{j}=a_{i}$ whenever $j \equiv i \pmod{e}$. The block $a_{1}a_{2}...a_{e}$ is called the period of $\frac{x}{N}$ and we write $\sum_{x = 1}^{\infty}\frac{a_i}{B^i}$ as $0.\overline{a_{1}a_{2}...a_{e}}$ or $0.\overline{a_{1}a_{2}...a_{e}}_{(B)}$ when necessary to indicate the base $B$. An important role will be played by the fact that the $a_i$ can be expressed in another way. For this we introduce a non-standard but useful notation. For any $z \in \Z$ there is a unique $y, 1 \leq y \leq N$ such that $z \equiv y\pmod{N}$ and we denote this $y$ as $ \langle z \rangle $; thus $z_{1} \equiv z_{2} \pmod{N}$ iff $ \langle z_{1} \rangle  =  \langle z_{2} \rangle$. 

\begin{lem} \label{lm1}
The $B$-digits $a_{1},a_{2},...$ in the base $B$ expansion of $\frac{x_1}{N}$ are given by

\begin{equation} \label{eq5}
a_{i}=\frac{B \langle B^{i-1}x_{1} \rangle - \langle B^{i}x_{1} \rangle }{N} ~.
\end{equation}
\end{lem}

\begin{proof}
We've seen $x_{i+1} \equiv B^{i}x_{1}\pmod{N}$ so $x_{i+1} =  \langle B^{i}x_{1} \rangle $ and similarly, $x_{i}= \langle B^{i-1}x_{1} \rangle $. So equation $i$ in the LDA becomes $B \langle B^{i-1}x_{1} \rangle =a_{i}N+ \langle B^{i}x_{1} \rangle $. Solving for $a_i$ proves the lemma.
\end{proof}

We call the sequence of the $e$ distinct numbers $x_{1},x_{2},...,x_{e}$ in the LDA a $B$-cycle, denoted as $C=(x_{1},x_{2},...,x_{e})$. Since the LDA for $\frac{x_{2}}{N}$ starts with equation $2$, one sees $\frac{x_{2}}{N}=0.\overline{a_{2}...a_{e}a_{1}}$ and so on. Thus $C = (x_{2},...,x_{e},x_{1})$ and any $x_{i}$ in the cycle can be chosen as the initial term. (Actually these cycles are just the permutation cycles for the permutation $x \rightarrow  \langle Bx \rangle $ on $X$). Since $|X|=\phi(N)$ and each cycle has $e$ numbers, the total number of cycles for $B$ on $X$ is $f=\frac{\phi(N)}{e}$. A numerical example may be useful here.

Let $N=15,$ $B=7$. The LDA for $\frac{1}{15}$ is

$$
\begin{array}{rcl}
7 \times 1 &=& 0 \times 15 + 7 \\
7 \times 7 &=& 3 \times 15 + 4 \\
7 \times 4 &=& 1 \times 15 + 13 \\
7 \times 13 &=& 6 \times 15 + 1 \\
\end{array} 
$$

Since $x_{5}=x_{1},e=4$ and $\frac{1}{15}=0.\overline{0316}_{(7)}$; the cycle containing $1$ is $C_{1}=(1,7,4,13)$. Starting with $x_{1}=14$ one finds $\frac{14}{15}=0.\overline{6350}_{(7)}$ and the cycle $C_{2}=(14,8,11,2)$.

After these preliminaries we return to the class number formula. Fix $D<-4,N=|D|,X$ the set of integers from $1$ to $N$ relatively prime to $N,h=h(D),\Chi=\Chi_{D}$. Choose a base $B>1$ prime to $N$ with $e$ being the order of $B$ mod $N$. The formula (\ref{eq1}) may now be written as $h=-\frac{1}{N} \sum_{x \in X}\Chi(x)x$. Let $C=(x_{1},x_{2},...,x_{e})$ be a cycle for $B$ on $X$. We isolate the contributions of $C$ to this formula for $h$ by defining 

\begin{equation} \label{eq6}
h_{C}=-\frac{1}{N}\sum_{x \in C}\Chi(x)x=-\frac{1}{N}\sum_{i=1}^{e}\Chi(x_i)x_i.
\end{equation}

$x_{i} \equiv B^{i-1}x_{1} \pmod{N}$ shows $\Chi(x_{i})=\Chi(B)^{i-1}\Chi(x_{1})$, and writing $x_{i}= \langle B^{i-1}x_{1} \rangle $, (\ref{eq6}) becomes

\begin{equation} \label{eq7}
h_{C}=-\frac{\Chi(x_{1})}{N}\sum_{i=1}^{e}\Chi(B)^{i-1} \langle B^{i-1}x_{1} \rangle .
\end{equation}

There are now two cases to consider depending on $\Chi(B)= \pm 1$. If $\Chi(B)= -1$ then $B^{e} \equiv 1\pmod{N}$ implies $1=\Chi(B^{e})=(-1)^{e}$, so $e$ is even. Since for any $i$, $x_{i+1} \equiv Bx_{i} \pmod{N}, \Chi(x_{i+1})=\Chi(B)\Chi(x_{i})=-\Chi(x_{i})$ so half the numbers in a cycle have $\Chi=1$ and half $\Chi=-1$. We now normalize $C$ by choosing the initial $x_{1}$ to have $\Chi(x_{1})=1$. Now (\ref{eq7}) becomes

\begin{equation} \label{eq8}
h_{C}=-\frac{1}{N}\sum_{i=1}^{e}(-1)^{i-1} \langle B^{i-1}x_{1} \rangle .
\end{equation}

For example, referring back to the example $N=15$, corresponding to $D=-15$, we see the cycle $C_1$ is normalized, but $C_2$ is not, since $\Chi(14)=-1$, as $\Chi_{-15}(14)=\left(\frac{14}{15}\right) = -1$. To normalize $C_2$ we set $C_{2}=(2,14,8,11),\Chi_{-15}(2)=\left(\frac{2}{15}\right) = 1$.

If $\Chi(B)=1$, then $x_{i+1} \equiv Bx_{i} \pmod{N}$ shows $\Chi(x_{i+1})=\Chi(B)\Chi(x_{i})=\Chi(x_{i})$ so all the numbers in a cycle have the same $\Chi$ value. We define $\Chi(C)=1$ if all $\Chi(x_{i})=1,\Chi(C)=-1$ if all $\Chi(x_{i})=-1$. In this case (\ref{eq7}) becomes

\begin{equation} \label{eq9}
h_{C}=-\frac{\Chi(C)}{N}\sum_{i=1}^{e} \langle B^{i-1}x_{1} \rangle .
\end{equation} 

Again using the previous example with $D=-15$ but with $B=4,\Chi_{-15}(4)=1$. One verifies easily that $e=2$ and there are $\frac{\phi(15)}{2}=4$ cycles for $B=4:$
$$
C_{1}=(1,4),C_{2}=(2,8),C_{3}=(7,13),C_{4}=(11,14)
$$
and 
$$
\Chi(C_{1})=\Chi(C_{2})=1, \Chi(C_{3})=\Chi(C_{4})=-1.
$$
Keeping all the previous notation, here is the main result of this section.

\begin{thm}
Let $C_{1},C_{2},...,C_{f}$ be the cycles for $B$ on $X$. Write $C_{j}=(x_{1}^{(j)},x_{2}^{(j)},...,x_{e}^{(j)}),1 \leq j \leq f$ and let $\frac{x_{1}^{(j)}}{N}=0. \overline{a_{1}^{(j)}a_{2}^{(j)}...{a_{e}^{(j)}}_{(B)}}$.

\begin{enumerate}
\item Case 1: $\Chi(B)=-1$. Assume all cycles $C_{j}$ normalized. Then

\begin{equation} \label{eq10}
(B+1)h(D)=\sum_{j=1}^{f}  \sum_{i=1}^{e}(-1)^{i}a_{i}^{(j)}
\end{equation} 

\item Case 2: $\Chi(B)=1$. Then

\begin{equation} \label{eq11}
(B-1)h(D)=-\sum_{j=1}^{f}\Chi(C_{j})  \sum_{i=1}^{e}a_{i}^{(j)}
\end{equation}
\end{enumerate} 
\end{thm}

\begin{proof}
When $\Chi(B)=-1$, $e$ is even and in (\ref{eq8}) both $(-1)^{i-1}$ and $ \langle B^{i-1}x_{1} \rangle $ have period $e$ so that (\ref{eq8}) can be written as $h_{C}=-\frac{1}{N}\sum_{i=1}^{e}(-1)^{i} \langle B^{i}x_{1} \rangle .$ On the other hand, multiply (\ref{eq8}) by $B$ and absorb the outside minus sign by replacing $(-1)^{i-1}$ by $(-1)^{i}$ to obtain $Bh_{C}=\frac{1}{N}\sum_{i=1}^{e}(-1)^{i}B \langle B^{i-1}x_{1} \rangle .$ Thus, $(B+1)h_{C}=Bh_{C}+h_{C}$
$$
\begin{array}{l}
= \frac{1}{N}\sum_{i=1}^{e}(-1)^{i}B \langle B^{i-1}x_{1} \rangle -\frac{1}{N}\sum_{i=1}^{e}(-1)^{i} \langle B^{i}x_{1} \rangle\\

= \sum_{i=1}^{e}(-1)^{i}\left(\frac{B \langle B^{i-1}x_{1} \rangle - \langle B^{i}x_{1} \rangle }{N}\right)\\

=\sum_{i=1}^{e}(-1)^{i}a_{i},\\ 
\end{array}
$$
by Lemma \ref{lm1}, if $\frac{x_{1}}{N}=0.\overline{a_{1}a_{2}...a_{e}}_{(B)}$. Now $h=\sum_{j=1}^{f}h_{C_{j}}$, so putting a superscript $(j)$ on the data for $C_j$ proves Case 1.

Now assume $\Chi(B)=1$. Since $B$ has period $e$, (\ref{eq9}) can be written as

$$
h_{C}=-\frac{\Chi(C)}{N}\sum_{i=1}^{e} \langle B^{i}x_{1} \rangle .
$$
On the other hand, multiply (\ref{eq9}) by $B$ to get $Bh_{C}=-\frac{\Chi(C)}{N}\sum_{i=1}^{e}B \langle B^{i}x_{1} \rangle $. Combining, $(B-1)h_{C}=Bh_{C}-h_{C}=-\Chi(C)\sum_{i=1}^{e}\frac{B \langle B^{i-1}x_{1} \rangle - \langle B^{i}x_{1} \rangle }{N}=-\Chi(C)\sum_{i=1}^{e}a_{i}$, by Lemma \ref{lm1}, where $\frac{x_{1}}{N}=0.\overline{a_{1}a_{2}...a_{e}}_{(B)}$. Since $h=\sum_{j=1}^{f}h_{C_{j}}$, putting a superscript $(j)$ on the data for $C_{j}$ proves Case 2 and completes the proof of the theorem.
\end{proof}

To illustrate the theorem consider again $D=-15$. With $B = 7,$ $e=4,\Chi(7)=-1$ we are in Case 1, the normalized cycles are $C_{1}=(1,7,4,13),$ $C_{2}=(2,14,8,11),\frac{1}{15}=0.\overline{0316}_{(7)},\frac{2}{15}=0.\overline{0635}_{(7)}$. The right side of (\ref{eq10}) is 
$$
\sum_{j=1}^{2}\sum_{i=1}^{4}(-1)^{i}a_{i}^{(j)}=(-0+3-1+6)+(-0+6-3+5)=16
$$ 
and the left side is $(7+1)h(-15)$. If one consults the table, or simply works out (\ref{eq1}) for this case, one finds $h(-15)=2$, confirming the theorem. Or one can consider this as a proof that $h(-15)=2$. Now take $B=4,e=2,\Chi(4)=1$,and the cycles $C_{1},$ $C_{2},C_{3},C_{4}$ as before, we are in Case 2. Now $\frac{1}{15}=0.\overline{01}_{(4)},\frac{2}{15}=0.\overline{02}_{(4)},\frac{7}{15}=0.\overline{13}_{(4)},\frac{11}{15}=0.\overline{23}_{(4)}$. The right side of (\ref{eq11}) is $-[(0+1)+(0+2)-(1+3)-(2+3)]=6$ and the left side is $(4-1)h(-15)=3 \times 2 = 6$.

Girstmair's proposition (\ref{eq2}) is a special case of the theorem. With $D=-p,N=p,X=\{1,2,...,p-1\},B$ a primitive root mod $p$ has order $e=p-1=\phi(N)$ so there is only one cycle $C=(1,...)$, which is normalized. We must have $\Chi(B)=-1$. For if $\Chi(B)=1$, since every $x$ in $X$ satisfies $x \equiv B^{k} \pmod{p}$, for some $k$, $\Chi(x)=\Chi(B)^{k}=1$. In particular $\Chi(p-1)=\Chi(-1)=1$ contra the property of $\Chi$ which says $\Chi(-1)=-1$. So we are in Case 1. Let $\frac{1}{p}=0.\overline{a_{1}a_{2}...a_{p-1}}_{(B)}$. Then by (\ref{eq10}), $(B+1)h(-p)=\sum_{i=1}^{p-1}(-1)^{i}a_{i}$, which is $(\ref{eq2})$.

\section{A new formula}

The results of the previous section, though interesting, have two drawbacks: they are not especially useful in calculating $h$, and the cases $\Chi(B)=1,\Chi(B)=-1$ have to be considered separately.

Keeping the previous notation, we note that a given $x \in X$ appears in exactly one cycle for $B$ on $X$, say as $x=x_{i}^{(j)}$ in the cycle $C_{j}$, normalized if necessary. Then in the LDA for $\frac{x_{1}^{(j)}}{N}$, the $i^{th}$ equation is $Bx_{i}^{(j)}=a_{i}^{(j)}N+ x_{i+1}^{(j)}$, where $a_{i}^{(j)}= \left[\frac{Bx_{i}^{(j)}}{N}\right] = \left[\frac{Bx}{N} \right].$ If $\Chi(B)=-1$, then in (\ref{eq10}) the coefficient of $a_{i}^{(j)}$ is $(-1)^{i}=\Chi(B)^{i}$, but $x=x_{i}^{(j)} \equiv B^{i-1}x_{1}^{(j)} \pmod{N}$ so that $\Chi(x)=\Chi(B)^{i-1}\Chi(x_{1}^{(j)})=(-1)^{i-1}$, since $\Chi(x_{1}^{(j)})=1$, by normalization. Thus $(-1)^{i}=-\Chi(x)$ is the coefficient of $a_{i}^{(j)}=\left[\frac{Bx}{N}\right]$ so the total contribution of the term $(-1)^{i}a_{i}^{(j)}$ is $-\Chi(x)\left[\frac{Bx}{N}\right]$. Since $B+1=B-\Chi(B)$, the formula (\ref{eq10}) becomes $(B-\Chi(B))h= -\sum_{x \in X}\Chi(x) \left[\frac{Bx}{N}\right]$. If $\Chi(B)=1$, then in (\ref{eq11}) the coefficient of $a_{i}^{(j)}$ is $-\Chi(C_{j})=-\Chi(x_{i}^{(j)})=-\Chi(x)$. Since $B-1=B-\Chi(B)$, (\ref{eq11}) becomes $(B-\Chi(B))h=-\sum_{x \in X}\Chi(x) \left[\frac{Bx}{N}\right]$. Thus in both cases (\ref{eq10}), (\ref{eq11}) are subsumed under the single formula

\begin{equation} \label{eq12}
-\sum_{x \in X}\Chi(x) \left[\frac{Bx}{N}\right] = (B-\Chi(B))h.
\end{equation}

Since $\left[\frac{Bx}{N}\right]$ is a $B$-digit we look to see when is $\left[\frac{Bx}{N}\right]=k$, for $0 \leq k \leq B-1$.

\begin{lem} \label{lm2}
Let $k$ be an integer, $0 \leq k \leq B-1$.  For $x \in X, \left[\frac{Bx}{N}\right] = k$ if and only if $\frac{kN}{B} < x < \frac{(k+1)N}{B}$.
\end{lem}

\begin{proof}
Since $\left[\frac{Bx}{N}\right]$ is never an integer, $\left[\frac{Bx}{N}\right]=k$ iff $k < \frac{Bx}{N} < k+1$; solving the inequality for $x$ proves the lemma.
\end{proof}

For $0 \leq k \leq B-1$ we denote the interval $\left(\frac{kN}{B},\frac{(k+1)N}{B}\right]$ on the real axis by $I_{k}$. These intervals, each of length $\frac{N}{B}$, form a partition of the interval $(0,N]$. By the above lemma, every $x$ is an interior point (not an endpoint) of exactly one $I_{k}$. We set $X_{k}= X \cap I_{k} = \left\{x : \frac{kN}{B} < x < \frac{(k+1)N}{B} \right\} = \left\{x : \left[\frac{Bx}{N}\right] = k \right\}$. Of course some of the sets $X_{k}$ may be empty. A point of notation. We are always assuming that $D$, hence $h,\Chi,N$, are given and fixed. However, the intervals $I_{k},X_{k}$ depend on $B$, and when necessary to indicate this we write $I_{k}(B),X_{k}(B)$. Now (\ref{eq12}) may be written as

\begin{equation} \label{eq13}
-\sum_{k=0}^{B-1}k \sum_{x \in X_{k}}\Chi(x)= (B-\Chi(B))h.
\end{equation}

For brevity we now define $E_{k}=\sum_{x \in X_{k}}\Chi(x)$. To show the dependence on $B$, we write $E_{k}(B)$. From now on if a sum is over $x$ we may not indicate this explicitly in the summation sign. Thus, $E_{k}=\sum_{\frac{kN}{B}}^{\frac{(k+1)N}{B}}\Chi(x)$ means sum over all values of $x$ between $\frac{kN}{B}$ and $\frac{(k+1)N}{B}$. Set $X_{k}^{+}= \left\{x \in X_{k} : \Chi(x)=1 \right\}$ and $X_{k}^{-}= \left\{x \in X_{k} : \Chi(x)=-1 \right\}$. Then we also have $E_{k}=|X_{k}^{+}|-|X_{k}^{-}|$. Equation (\ref{eq13}) now becomes
\begin{equation} \label{eq14}
-\sum_{k=0}^{B-1}kE_{k}(B) = (B-\Chi(B))h.
\end{equation}
and we use this to state our main result.

\begin{thm}\label{thm2}
\begin{equation} \label{eq15}
\sum_{k=0}^{\left[\frac{B}{2}\right]-1}(B-1-2k)E_{k}(B) = (B-\Chi(B))h.
\end{equation}
If $B=B_{1}B_{2}$ is a proper factorization of $B,1 < B_{1} < B$, then
\begin{equation} \label{eq16}
\sum_{k=0}^{\left[\frac{B_1}{2}\right]-1}(B_{1}-1-2k) \sum_{j=0}^{B_{2}-1}E_{kB_{2}+j}(B) = (B_{1}-\Chi(B_{1}))h.
\end{equation}
\end{thm}

\noindent \textit{Remark}. Equation (\ref{eq15}) may be considered as included in (\ref{eq16}) if one sets $B_{1}=B,B_{2}=1$.

\begin{proof}
Consider the map $\xi(x)=N-x$. It is easily seen that $\xi$ is a permutation of $X,\xi$ has no fixed points in $X$ and is an involution: $\xi^{2}$ is the identity on $X$. Also $\Chi(\xi(x))=\Chi(N-x)=\Chi(-x)=-\Chi(x)$ so $x$ and $\xi(x)$ have opposite $\Chi$ values. If $x \in X_{k},\frac{kN}{B} < x < \frac{(k+1)N}{B}$, then $\frac{(B-1-k)N}{B} < N-x < \frac{(B-k)N}{B}$. We define $\gamma$ on the set of $B$-digits $\{0,1,...,B-1\}$ by $\gamma(k) = B-1-k$, which is a permutation of the set of $B$-digits, also an involution. Thus, if $x \in X_{k}$ and $\gamma(k)=k'$, then $\xi(x) \in X_{k'}$. So $\xi$ is a bijection of $X_{k}$ onto $X_{k'}$, but since $\xi$ interchanges $\Chi$ values, $\xi$ maps $X_{k}^{+}$ onto $X_{k'}^{-}$ and $X_{k}^{-}$ onto $X_{k'}^{+}$. Hence, $E_{k'}(B)=|X_{k'}^{+}|-|X_{k'}^{-}|=|X_{k}^{-}|-|X_{k}^{+}|=-E_{k}(B)$. In particular, if $B$ is odd then $\frac{B-1}{2}$ is a $B$-digit and $\gamma(\frac{B-1}{2})=\frac{B-1}{2}$ so $E_{\frac{B-1}{2}}(B)=0$. Whether $B$ is odd or even, the left side of (\ref{eq14}) is $-\sum_{_{1}}-\sum_{_{2}}$ where $\sum_{_{1}}=\sum_{0 \leq k < \frac{B-1}{2}}kE_{k}(B)$ and $\sum_{_{2}}=\sum_{\frac{B-1}{2} < k \leq B-1}kE_{k}(B)$. In $\sum_{_{2}}$ make the change of variable $k= B-1-j$ to obtain $\sum_{_{2}}=\sum_{0 \leq j < \frac{B-1}{2}}(B-1-j)E_{B-1-j}(B)=\sum_{0 \leq j < \frac{B-1}{2}}(B-1-j)E_{j'}(B)$, where $j'=\gamma(j)$. But $E_{j'}(B)=-E_{j}(B)$, so $\sum_{_{2}}=-\sum_{0 \leq j < \frac{B-1}{2}}(B-1-j)E_{j}(B)$. In this last sum we rename the dummy index $j$ to be $k$ and combining it with $\sum_{_{1}}$ yields $-\sum_{_{1}}-\sum_{_{2}}= -\sum_{0 \leq k < \frac{B-1}{2}}kE_{k}(B)+\sum_{0 \leq k < \frac{B-1}{2}}(B-1-k)E_{k}(B)=\sum_{0 \leq k < \frac{B-1}{2}}(B-1-2k)E_{k}(B)$. Thus, (\ref{eq14}) now becomes $\sum_{0 \leq k < \frac{B-1}{2}}(B-1-2k)E_{k}(B)=(B-\Chi(B))h$. Let $g$ be the largest integer $< \frac{B-1}{2}$. If $B$ is even $=2n, \frac{B-1}{2}=n-\frac{1}{2}$, so $g=n-1=\left[\frac{B}{2}\right]-1$. If $B$ is odd = $2n+1, \frac{B-1}{2}=n,$ so $g=n-1=\left[\frac{B}{2}\right]-1$. So in either case $\sum_{0 \leq k < \frac{B-1}{2}}=\sum_{k=0 }^{\left[\frac{B}{2}\right]-1}$, which proves (\ref{eq15}).

Now suppose $B=B_{1}B_{2},1 < B_{1} < B$. With $B_{1}$ in place of $B$, (\ref{eq15}) shows 

$$
\sum_{k=0 }^{\left[\frac{B_{1}}{2}\right]-1}(B_{1}-1-2k)E_{k}(B_{1})=(B_{1}-\Chi(B_{1}))h.
$$ 

When the interval $(0,N]$ is divided into the $B$ intervals $I_{k}(B)$, each interval has length $\frac{N}{B}$, while with the smaller $B_{1}$ one obtains $B_{1}$ intervals $I_{k}(B_{1})$  each of greater length $\frac{N}{B_{1}}$. How are these intervals related? Since $B_{1}=\frac{B}{B_{2}}$, $I_{k}(B_{1}) = $
$$
\begin{array}{ll}
 \displaystyle \left(\frac{kN}{B_{1}},\frac{(k+1)N}{B_{1}}\right] 
&= \displaystyle  \left(\frac{kB_{2}N}{B},\frac{(k+1)B_{2}N}{B}\right]\\
&= \displaystyle \bigcup_{j=0}^{B_{2}-1} \left(\frac{(kB_{2}+j)N}{B},\frac{(kB_{2}+j+1)N}{B}\right]\\
&=  \displaystyle \bigcup_{j=0}^{B_{2}-1}I_{kB_{2}+j}(B).
\end{array}
$$
Thus $E_{k}(B_{1})=\sum_{x \in I_{k}(B_{1})}\Chi(x)=\sum_{j=0}^{B_{2}-1}E_{kB_{2}+j}(B)$. Substituting this last sum for $E_{k}(B_{1})$ in (\ref{eq15}) as stated for $B_{1}$ proves (\ref{eq16}) and the proof of the theorem is complete. 
\end{proof}

The applications of this theorem are covered in the next two sections. The cases $D \equiv 1\pmod{4})$ and $D \equiv 0 \pmod{4}$ must be treated separately. Here we make only a general comment on the method involved. For a given $B$, (\ref{eq15}) involves the $\left[\frac{B}{2}\right]$ quantities $E_{k}(B),0 \leq k \leq \left[\frac{B}{2}\right]-1$. Let $d(B)$ denote the number of divisors $B_{1}$ of $B$. For each $B_{1} > 1$ there is an equation (\ref{eq16}) involving the quantities $E_{k}(B)$. So we have a system of $d(B)-1$ linear equations for the $\left[\frac{B}{2}\right]$ unknowns. If $\left[\frac{B}{2}\right] \leq d(B)-1$ one can expect (or hope) to find a unique solution to the system. This occurs for $B=2,3,4,6$, where equality holds and the program succeeds. There does not appear to be any other $B$ where the equality holds. For $B=12, \left[\frac{12}{2}\right]=6,d(B)-1=5$ and we have $5$ equations for $6$ unknowns. A unique solution is not found, but some partial information is obtained; beyond $B=12$ we have not ventured.

\section{$D \equiv 1\pmod{4}$}

With $D$ being odd, one can choose $B=2$; (\ref{eq15}) then has only one term (for $k=0$) and yields $E_{0}(2)=(2-\Chi(2))h$. But $\Chi(2)=\left(\frac{2}{N}\right)$ is $1$ or $-1$ according, as $N \equiv 7\pmod{8}$ or $N \equiv 3\pmod{8}$. Thus
$$
E_0(2) = \sum_{0}^{\frac{N}{2}} = \left\{
\begin{array}{ll}
h; &\mbox{if}\; N \equiv 7 \pmod{8}\\
3h; &\mbox{if}\; N \equiv 3 \pmod{8} 
\end{array}
\right.
$$

This result appears already in \cite{BoSha}, p. 346, where it is derived by manipulation of the basic formula (\ref{eq1}), relevant only for $B=2$. However, it has an important consequence. If $p>3$ is a prime and $p\equiv3 \pmod{4}$, then $E_{0}(2)=|X_{0}^{+}(2)|-|X_{0}^{-}(2)|$ is the number of quadratic residues minus the number of quadratic non-residues in the interval $(0,\frac{p}{2})$. Since $h$ is a positive integer, this shows that the residues always outnumber the non-residues in this interval. Apparently, there is no direct proof of this fact by the methods of ``elementary" number theory and this is a triumph of the class number formula. This result can now be refined. Take $B=4$; then there are two equations from $(\ref{eq16})$ for $B_{1}=2$ and $B_{1}=4$ (recall the remark after the statement of Theorem \ref{thm2}). They are
\begin{equation*}
\begin{array}{l}
 \text{for}~ B_{1}=2: \sum_{k=0}^{0}(2-1-2k)\sum_{j=0}^{1}E_{j}(4)=(2-\Chi(2))h \\
 \text{for}~ B=4: \sum_{k=0}^{1}(4-1-2k)E_{k}(4)=(4-\Chi(4))h. \\
\end{array} 
\end{equation*}

Since $h>0$, define $y_{k}=y_{k}(B)=\frac{E_{k}(B)}{h}$, and we have the system
$$y_{0}+y_{1}=2-\Chi(2)$$
$$3y_{0}+y_{1}=4-\Chi(4)$$
Noting the values of $\Chi(2)$ discussed above, and $\Chi(4)=1$, the system is easily seen to show

\begin{thm}
With $E_{0}(4)=\sum_{0}^{\frac{N}{4}}\Chi(x),~ E_{1}(4)=\sum_{\frac{N}{4}}^{\frac{N}{2}}\Chi(x)$, then 
\begin{center}

for $N \equiv 7 \pmod{8}$, $E_{0}(4)=h,~E_{1}(4)=0$  

for $N \equiv 3 \pmod{8},$ $ E_{0}(4)=0,~E_{1}(4)=3h.$ 
\end{center}
\end{thm} 

\hfill $\Box$

Here are two numerical examples:

\begin{equation}\label{eq17}
D=-39 \equiv 1 \pmod{8},~ N=39 \equiv 7 p\mod{8}
\end{equation}

\begin{center}
\scalebox{0.9}{
\begin{tabular}{r| r r r r r r c r r r r r r c}
 $$& $$& $$& $$& $$& $$& $$& $\frac{N}{4}$& $$& $$& $$& $$& $$& $$& $\frac{N}{2}$ \\
 $x$ & $1$ &$2$& $4$ &$5$& $7$  &$8$& $\uparrow$   &$10$& $11$  &$14$& $16$ & $17$ & $19$ & $\uparrow$ \\ \hline
 $\Chi(x)$ & $1$ &$1$& $1$ &$1$& $-1$  &$1$& $$   &$1$& $1$  &$-1$& $1$ & $-1$ & $-1$ & $$ \\
\end{tabular}
}

$$E_{0}(4)=4,~h(-39)=4,~E_{1}(4)=0;$$

\end{center}

\begin{equation}\label{eq18}
D=-43 \equiv 5 \pmod{8},~ N=43 \equiv3 \pmod{8}
\end{equation}

\begin{center}
\scalebox{0.7}{
\begin{tabular}{r| r r r r r r r r r r c r r r r r r r r r r r c}
$$& $$& $$& $$& $$& $$& $$& $$& $$& $$& $$& $\frac{N}{4}$& $$& $$& $$& $$& $$& $$& $$& $$& $$& $$& $$& $\frac{N}{2}$ \\
 $x$ & $1$ &$2$& $3$ & $4$ &$5$& $6$ & $7$  &$8$& $9$ & $10$ & $\uparrow$   &$11$ &$12$ &$13$  &$14$& $15$ &$16$ & $17$ & $18$ & $19$ & $20$ &$21$ & $\uparrow$ \\ \hline
 $\Chi(x)$ & $1$ &$-1$& $-1$ & $1$ &$-1$& $1$ & $-1$  &$-1$& $1$ & $1$ & $$   &$1$ &$-1$ &$1$  &$1$& $1$ &$1$ & $1$ & $-1$ & $-1$ & $-1$ &$1$ & $$ \\
\end{tabular}
}

$$
E_0(4) = 0, E_1(4)=3, h(-43) = 1.
$$

\end{center}

Assume now $3 \not| D$. Then $B=6$ is prime to $D$ and there are three equations available from $B_{1}=2,~B_{1}=3,~B_{1}=B=6$ and there are three unknowns $E_{0}(6),~E_{1}(6),~E_{2}(6).$ Following the same procedure as before, there is a linear system,

\begin{equation*}
\begin{array}{l}
 y_{0}+y_{1}+y_{2}=2-\Chi(2) \\
 2y_{0}+2y_{1}=3-\Chi(3) \\
 5y_{0}+3y_{1}+y_{2}=6-\Chi(6) \\
\end{array} 
\end{equation*}

The coefficient matrix 
$$
\begin{pmatrix} 1 & 1 & 1 \\ 2 & 2 & 0 \\ 5 & 3 & 1 \end{pmatrix}
$$ 
has determinant $-4$. Let $a=2-\Chi(2),~b=3-\Chi(3),~c=6-\Chi(6)$ and solve by Cramer's rule to obtain $y_{0}=\frac{1}{2}(-a-b+c),~y_{1}=\frac{1}{2}(a+2b-c),~y_{2}=\frac{1}{2}(2a-b)$. What are $a,b,c$? We've already discussed $\Chi(2)$. Now $\Chi(3)=\left(\frac{3}{N}\right)=-\left(\frac{N}{3}\right)$, since $N \equiv 3 \pmod{4}$, and $\left(\frac{N}{3}\right)= 1$ or $-1$ according as $N \equiv 1$ or $2 \pmod{3}$. Altogether there are $4$ cases:

$$\text{Case}~ 1: \begin{Bmatrix} \Chi(2)=1 \\ \Chi(3)=1 \end{Bmatrix} = \begin{Bmatrix} N \equiv 7 \pmod{8} \\ N \equiv 2\pmod{3} \end{Bmatrix} \iff N \equiv 23 \pmod{24}$$ 

$$\text{Case}~ 2: \begin{Bmatrix} \Chi(2)=-1 \\ \Chi(3)=1 \end{Bmatrix} = \begin{Bmatrix} N \equiv 3 \pmod{8} \\ N \equiv 2 \pmod{3} \end{Bmatrix} \iff N \equiv 11 \pmod{24}$$ 

$$\text{Case}~ 3: \begin{Bmatrix} \Chi(2)=1 \\ \Chi(3)=-1 \end{Bmatrix} = \begin{Bmatrix} N \equiv 7 \pmod{8} \\ N \equiv 1 \pmod{3} \end{Bmatrix} \iff N \equiv 7 \pmod{24}$$ 

$$\text{Case}~ 4: \begin{Bmatrix} \Chi(2)=-1 \\ \Chi(3)=-1 \end{Bmatrix} = \begin{Bmatrix} N \equiv 3 \pmod{8} \\ N \equiv 1\pmod{3} \end{Bmatrix} \iff N \equiv 19 \pmod{24}$$ 

In terms of $D$, these correspond to $D \equiv 1,13,17,5 \pmod{24}$ and any $D \equiv 1\pmod{4}$ not divisible by $3$ is in one of these congruence classes. Evaluating $a,b,c$ for each case and then $y_{0},y_{1},y_{2}$ one finds:

\begin{tabular}{ l l l l l l l }
  Case 1: & $a$=1, & $b$=2, & $c$=5;  & $y_{0}$=1,& $y_{1}$=0,& $y_{2}$=0 \\
  Case 2: & $a$=3, & $b$=2, & $c$=7;  & $y_{0}$=1,& $y_{1}$=0,& $y_{2}$=2 \\
  Case 3: & $a$=1, & $b$=4, & $c$=7;  & $y_{0}$=1,& $y_{1}$=1,& $y_{2}$=-1 \\
  Case 4: & $a$=3, & $b$=4, & $c$=5;  & $y_{0}$=-1,& $y_{1}$=3,& $y_{2}$=1 \\
\end{tabular}

Since $y_{k}=\frac{E_{k}}{h}$, we have the following result.

\begin{thm}\label{4pt2}
Assume $3$ does not divide $D$. Then for
 
\begin{tabular}{ lllll}
$N \equiv 23 \pmod{24}$: & $E_{0}(6)=h,$ & $E_{1}(6)=0$,    &$E_{2}(6)=0$ \\
$N \equiv 11\pmod{24}$: & $E_{0}(6)=h,$ &$E_{1}(6)=0$,     &$E_{2}(6)=2h$ \\
$N \equiv 7 \pmod{24}$:   &  $E_{0}(6)=h,$  &$E_{1}(6)=h,$   &$E_{2}(6)=-h$ \\
$N \equiv 19 \pmod{24}$: & $E_{0}(6)=-h,$ &$E_{1}(6)=3h,$ &$E_{2}(6)=h.$ \\
\end{tabular}

\end{thm}

\begin{cor}\label{cor43}
In all four cases, $h(D)=\left|\sum_{0}^{\frac{N}{6}}\Chi(x)\right|$. 
\end{cor}

\begin{proof}
Obvious by the previous theorem.
\end{proof}

For an illustration of the case $N \equiv 19 \pmod{24}$ one may return to (\ref{eq18}), the table shown before for $D=-43,N=43$, put markers between $7$ and $8$ for $\frac{N}{6}$, between $14$ and $15$ for $\frac{2N}{6}$. Then one sees $E_{0}(6)=-1=-h(-43),E_{1}(6)=3=3h(-43)$ and $E_{2}(6)=1=h(-43).$

Continuing with $3 \not| D$, consider $B=12$. As noted earlier here one here has a system of $5$ linear equations, corresponding to $B_{1}=2,B_{1}=3,B_{1}=4,B_{1}=6,B_{1}=B=12,$ for the six quantities $E_{k}(12), 0 \leq k \leq 5.$ Setting $y_{k}=\frac{E_{k}(12)}{h}$ , the equations are

\begin{tabular}{r r r r r r r r r r r r r}
 $y_{0}$ &+& $y_{1}$ &+& $y_{2}$     &+& $y_{3}$  &+& $y_{4}$   &+& $y_{5}$  &=& $2-\Chi(2)$ \\
 $2y_{0}$ &+& $2y_{1}$ &+& $2y_{2}$   &+& $2y_{3}$  &&         &&       &=& $3-\Chi(3)$ \\
 $3y_{0}$ &+& $3y_{1}$ &+& $3y_{2}$  &+& $y_{3}$   &+& $y_{4}$  &+& $y_{5}$  &=& $4-\Chi(4)$ \\
 $5y_{0}$ &+& $5y_{1}$ &+& $3y_{2}$  &+& $3y_{3}$  &+&  $y_{4}$  &+& $y_{5}$  &=& $6-\Chi(6)$ \\
 $11y_{0}$&+& $9y_{1}$ &+& $7y_{2}$&+& $5y_{3}$&+& $3y_{4}$&+& $y_{5}$&=&$12-\Chi(12)$ \\
\end{tabular}

For $N \equiv 23 \pmod{24})$ all the $\Chi$ values are $1$, so the constants on the right are $1,2,3,5,11$. By suitable elimination, one has $y_{1}=1-y_{0},y_{2}=0,y_{3}=0,y_{4}=1-y_{0},y_{5}=-1+y_{0}.$ Thus $E_{1}(12)=h-E_{0}(12),E_{2}(12)=0,E_{3}(12)=0,E_{4}(12)=h-E_{0}(12),E_{5}(12)=-h+E_{0}(12)$.

So unlike in Theorem \ref{4pt2}, where knowledge of only one of $h, E_{0}(6)$ is sufficient to determine the remaining items, here both $h$ and $E_{0}(12)$ are required to determine the remaining $E_{k}(12).$ For the remaining classes of $N \pmod{24}$, a similar elimination process can be carried out; details are left to the interested reader. Here we summarize the final results.

\begin{thm}
Assume $3 \not| D$. Once $h$ and $E_{0}=E_{0}(12)$ have been found, the remaining $E_{k}(12)$ are as follows:

\begin{tabular}{|l| r| r| r| r| r| r|}
\hline
$$ & $E_{1}(12)$ & $E_{2}(12)$ & $E_{3}(12)$ & $E_{4}(12)$ & $E_{5}(12)$ \\ \hline
$N \equiv 23(mod~24)$ & $h-E_{0}$ & $0$ & $0$ & $h-E_{0}$ & $-h+E_{0}$ \\ \hline
$N \equiv 11(mod~24)$ & $h-E_{0}$ & $-h$ & $h$ & $h-E_{0}$ & $h+E_{0}$ \\ \hline
$N \equiv 7(mod~24)$ & $h-E_{0}$ & $0$ & $h$ & $-E_{0}$ & $-h+E_{0}$ \\ \hline
$N \equiv 19(mod~24)$ & $-h-E_{0}$ & $h$ & $2h$ & $2h-E_{0}$ & $-h+E_{0}$ \\ \hline
\end{tabular}

\hfill $\Box$

\end{thm}

Again take (\ref{eq18}), the table for $N=43 \equiv 19 \pmod{24}$, and insert markers for $\frac{N}{12}$ between $3$ and $4$, for $\frac{2N}{12}$ between $7$ and $8$, for $\frac{3N}{12}$ between $10$ and $11$, for $\frac{4N}{12}$ between $14$ and $15$ and for $\frac{5N}{12}$ between $17$ and $18$. With $h(-43)=1$ and $E_{0}(12)=-1$ one sees $E_{1}(12)=0=-h-E_{0},~E_{2}(12)=1=h,~E_{3}(12)=2=2h,~E_{4}(12)=3=2h-E_{0},~E_{5}(12)=-2=-h+E_{0}$.

It is interesting to note that without knowing $h$ or $E_{0}$ one knows some of the other values, for example when a $0$ occurs in the table. Also the values in the columns $E_{2}(12),~E_{3}(12)$ depend only on $h$.

\section{$D \equiv 0 \pmod{4}$}

Now use of even $B$ is ruled out. In this case, however, it will be seen that there are new symmetries on the set $X$ which do not occur when $D$ is odd. We recall the three types of $\Chi_{D}$ listed in the Introduction. In all of them $m,n$ are negative square-free integers.

\begin{enumerate}
\item[(D1)] $D=4m,~m \equiv 3 \pmod{4},~\Chi_{D}(x)=\Chi_{4}(x)\left(\frac{x}{|m|}\right)$ \label{D1}\label{D1}
\item[(D2)] $D=4m,~m=2n,~n \equiv 1 \pmod{4},~\Chi_{D}(x)=\Chi_{8}(x)\left(\frac{x}{|n|}\right)$ \label{D2}
\item[(D3)] $D=4m,~m=2n,~n \equiv 3 \pmod{4},~\Chi_{D}(x)=\Chi_{4}(x)\Chi_{8}(x)\left(\frac{x}{|n|}\right)$ \label{D3}
\end{enumerate}

In (D1), $D \equiv 4 \pmod{8}$, while in (D2) and (D3), $D \equiv 0 \pmod{8}$.

We will need the following facts which follow immediately from their definitions. For $x$ odd, $u$ even,

$$\Chi_{4}(x+u)=\Chi_{4}(x) ~\text{if}~ u \equiv 0 \pmod{4}$$
$$\text{and}$$
$$\Chi_{4}(x+u)=-\Chi_{4}(x) ~\text{if}~ u \equiv 2 \pmod{4}.$$

$$ \Chi_{8}(x+u)=\Chi_{8}(x) ~\text{if}~ u \equiv 0 \pmod{8}$$
$$\text{and}$$
$$\Chi_{8}(x+u)=-\Chi_{8}(x) ~\text{if}~ u \equiv 4 \pmod{8}.$$

As usual, $N=|D|,~X$ is the set of integers $x,1 \leq x \leq N$ and $gcd(x,N)=1$. Since $N$ is now even, all $x$ are odd. We break up $X$ into two parts: $L$, the numbers to the left of $\frac{N}{2}$, and $R$, the numbers to the right of $\frac{N}{2}$; $L=\left\{x : x < \frac{N}{2} \right\}$, $R=\left\{x : x > \frac{N}{2}\right\}$. Besides $\xi(x)=N-x$, which clearly interchanges $L$ and $R$, the set $X$ has another permutation $\eta$ defined by 
$$
\eta(x) = \left\{
\begin{array}{ll}
x + \frac{N}{2}; &\mbox{if} \; x\in L\\
x  - \frac{N}{2}; &\mbox{if}\; x \in R\\
\end{array}
\right.
$$
$\eta$ also is an involution, $\eta^{2}(x)=x$ and $\eta$ interchanges $L$ and $R$. Like $\xi$, $\eta$ also interchanges $\Chi$ values: $\Chi(\eta(x))=-\Chi(x)$. To show this we consider case by case. 

If (D1), $\Chi_{D}(\eta(x))=\Chi_{4}(\eta(x))\left(\frac{\eta(x)}{|m|}\right), \eta(x)= x \pm \frac{N}{2}=x \pm 2|m|$ and $|m| \equiv 1\pmod{4}$ so $\pm2|m| \equiv 2\pmod{4}$ and $\Chi_{4}(\eta(x))=\Chi_{4}(x \pm 2|m|)=-\Chi_{4}(x)$, but $\left(\frac{\eta(x)}{|m|} \right)=\left(\frac{x \pm 2|m|}{|m|} \right)=\left(\frac{x}{|m|} \right)$, showing here $\Chi(\eta(x))=-\Chi(x)$. 

In (D2), (D3), $N=8|n|, \frac{N}{2}=4|n|$, so $\Chi_{4}(x \pm \frac{N}{2})=\Chi_{4}(x),~  \left(\frac{x \pm 4|n|}{|n|} \right)=  \left(\frac{x}{|n|} \right)$ but $\Chi_{8}(x \pm \frac{N}{2})=\Chi_{4}(x \pm 4|n|)=-\Chi_{8}(x)$, since $n$ is odd, $4|n| \equiv 4\pmod{8}$.

We now claim $\xi, \eta$ commute: $\xi\eta=\eta\xi$.

Proof by direct computation. 

$$\text{If}~ x \in L,~ \xi\eta \left(x \right)=\xi \left(x+\frac{N}{2} \right)=N- \left(x+\frac{N}{2} \right)=\frac{N}{2}-x$$

and

$$\eta\xi \left(x \right)=\eta \left(N-x \right)=  \left(N-x \right)-\frac{N}{2}  ~~\left(\text{since}~ N-x \in R \right)~ = \frac{N}{2}-x.$$

$$\text{If}~ x \in R,~ \xi\eta \left(x \right)=\xi \left(x-\frac{N}{2} \right)=N- \left(x-\frac{N}{2} \right)=\frac{3N}{2}-x$$

and

$$\eta\xi \left(x \right)=\eta \left(N-x \right)=  \left(N-x \right)+\frac{N}{2}  ~~\left(\text{since}~ N-x \in L \right)~ = \frac{3N}{2}-x.$$

Define $\lambda=\xi\eta=\eta\xi$. Then, clearly, $\lambda$ preserves $\Chi$ values, $\Chi(\lambda(x))=\Chi(x)$,

$$
\lambda(x) = \left\{
\begin{array}{ll}
\frac{N}{2}-x; &\text{if}\;  x \in L \\
\frac{3N}{2}-x; &\text{if}\;  x \in R \\
\end{array}
\right. \; \mbox{and $\lambda$ preserves $L$ and $R$}.
$$

In fact, $\lambda |_{L}$ ($\lambda$ restricted to $L$) is a reflection in $\frac{N}{4}$. Because if $x \in L$, write $x=\frac{N}{4}+y,~|y| < \frac{N}{4},~ \lambda(x)=\frac{N}{2}-(\frac{N}{4}+y)=\frac{N}{4}-y$. In the same way one sees that $\lambda |_{R}$ is a reflection in $\frac{3N}{4}.$ To help see the picture, here is an example. Let $D=-40=4(-10),~-10=2 \times (-5),~ -5 \equiv 3 \pmod{4}$. So $-40$ is $(D3),~ \Chi_{-40}(x)=\Chi_{4}(x)\Chi_{8}(x)\left(\frac{x}{5}\right)$. We tabulate the values for $x \in X$.

\begin{equation}\label{eq19}
\scalebox{0.75}{
\begin{tabular}{r| r r r r c r r r r c r r r r c r r r r c}
& & & & & $\frac{N}{4}$& & & & & $\frac{N}{2}$& & & & & $\frac{3N}{4}$& & & & & $N$ \\
$x$& $1$& $3$& $7$& $9$& $\uparrow$& $11$& $13$& $17$& $19$& $\uparrow$& $21$& $23$& $27$& $29$& $\uparrow$& $31$& $33$& $37$& $39$& $\uparrow$ \\
\hline
$\Chi(x)$& $1$& $-1$& $1$& $1$& $$& $1$& $1$& $-1$& $1$& $$& $-1$& $1$& $-1$& $-1$& $$& $-1$& $-1$& $1$& $-1$& $$ \\ \hline
 & & & & &$\rightarrow \lambda \leftarrow$ & & & & &$\rightarrow \xi \leftarrow$ & & & & &$\rightarrow \lambda \leftarrow$ & & & & & \\
\end{tabular}}
\end{equation}

The values of $\Chi_{-40}(x)$ for $x \in L$ were calculated from the definition. Now $\eta$ maps $L$ on $R$, changing $\Chi$ values so the values $\Chi(x)$ for $x \in R$ are found by listing those for $1,3,...,19$ in $L$ under $21,...,39$ with a change of sign. The $\lambda$ with arrows under the marker $\frac{N}{4}$ indicates the action of $\lambda$ on $L$ as a reflection through $\frac{N}{4}$, and similarly, the $\lambda$ with arrows under the marker $\frac{3N}{4}$ indicates the action of $\lambda$ on $R$ as a reflection through $\frac{3N}{4}$. In both cases, the reflections preserve the $\Chi$ values. On the other hand, writing any $x$ as $x=\frac{N}{2}+y, |y| < \frac{N}{2}$, one has $\xi(x)=N-x=N- \left( \frac{N}{2}+y \right)=\frac{N}{2}-y$, so $\xi$ is a reflection on $X$ through the point $\frac{N}{2}$, interchanging $L$ and $R$, and also changing the $\Chi$ values, as indicated by the $\xi$ with arrows.

\begin{lem}\label{lm3}
$$h(D)=\sum_{1}^{\frac{N}{4}}\Chi(x).$$
\end{lem}

\begin{proof}
By the basic class number formula (\ref{eq1}), $-Nh=\sum_{1}^{N}\Chi(x)x$ =$~\sum_{_{1}} + \sum_{_{2}}$, where $\sum_{_{1}}$ is the sum over $x \in L$ and $\sum_{_{2}}$ is the sum over $x \in R$. In $\sum_{_{2}}$, make the substitution $x=\eta(y)=y+\frac{N}{2}$ for $y \in L$, so $\sum_{_{2}}$$=-\sum_{y \in L}\Chi(y)\left(y+\frac{N}{2}\right)$, since $\Chi(\eta(y))=-\Chi(y)$. Thus $\sum_{_{2}}$$=-\sum_{y \in L}\Chi(y)y-\frac{N}{2}\sum_{y \in L}\Chi(y)=$ $-\sum_{_{1}}$ $-\frac{N}{2}\sum_{y \in L}\Chi(y)$, and the $\sum_{_{1}}$ sums cancel out, leaving $-Nh--\frac{N}{2}\sum_{y \in L}\Chi(y)$. But $\sum_{y \in L}\Chi(y)=\sum_{1}^{\frac{N}{2}}\Chi(y)=\sum_{1}^{\frac{N}{4}}\Chi(y)+\sum_{\frac{N}{4}}^{\frac{N}{2}}\Chi(y)$ and this last sum is, setting $y=\lambda(x)$, $\sum_{1}^{\frac{N}{4}}\Chi(\lambda(x))=\sum_{1}^{\frac{N}{4}}\Chi(x)$, since $\lambda$ preserves the $\Chi$ values. So $-Nh=-\frac{N}{2} \left( 2\sum_{1}^{\frac{N}{4}}\Chi(x)  \right)$, which proves the lemma.
\end{proof}

This result can be refined if we assume $3 \not| D$.

\begin{thm}
Assume D is not divisible by $3$. 
\begin{equation*}
\begin{array}{l}
 \text{If}~ D \equiv 1 \pmod{3},~\text{then}~ h=\sum_{1}^{\frac{N}{6}}\Chi(x),~\sum_{\frac{N}{6}}^{\frac{N}{4}}\Chi(x)=0\\
 \text{If}~ D \equiv 2 \pmod{3},~\text{then}~ \sum_{1}^{\frac{N}{6}}\Chi(x)=0,~h=\sum_{\frac{N}{6}}^{\frac{N}{4}}\Chi(x).\\
\end{array} 
\end{equation*}

\end{thm}

\begin{proof}
We can take $B=3$ and (\ref{eq15}) in Theorem \ref{thm2} gives $2E_{0}(3)=(3-\Chi(3))h$. We claim $\Chi(3)=\left( \frac{D}{3} \right)$. The proof is by considering the cases (D1), (D2), and (D3).

For (D1), $\Chi_{4}(3)\left(\frac{3}{|m|}\right)=-\left(\frac{3}{|m|}\right)$. Here $|m| \equiv 1\pmod{4}$, so $\left(\frac{3}{|m|}\right)=\left(\frac{|m|}{3}\right)$ and $\Chi(3)=-\left(\frac{|m|}{3}\right)=\left(\frac{m}{3}\right)=\left(\frac{4m}{3}\right)=\left(\frac{D}{3}\right)$. In case (D2), $\Chi(3)=\Chi_{8}(3)\left(\frac{3}{|n|}\right)=-\left(\frac{3}{|n|}\right)=-\left(- \left(\frac{|n|}{3}\right) \right)=\left(\frac{|n|}{3}\right)$, since here $|n| \equiv 3 \pmod{4}$. But $D = 8n \equiv -n = |n| \pmod{3}$, so $\Chi(3)=\left( \frac{D}{3} \right)$. In case (D3), $\Chi(3)=\Chi_{4}(3)\Chi_{8}(3)\left(\frac{3}{|n|}\right)=(-1)(-1)\left(\frac{|n|}{3}\right)$, since here $|n| \equiv 1\pmod{4}$. Again $D = 8n \equiv -n=|n| \pmod{3}$, so $\Chi(3)=\left(\frac{D}{3}\right)$.

Thus , $\Chi(3)=1$ if $D \equiv 1 \pmod{3}$ and $\Chi(3)=-1$ if $D \equiv 2 \pmod{3}$. 

So, $E_{0}(3)=\left(\frac{3-\Chi(3)}{2}\right)h= \left\{
\begin{array}{ll}
h, &\text{if}~ D \equiv1(mod~3) \\
2h, &\text{if}~ D \equiv2(mod~3).
\end{array}
\right.$

But also $E_{0}(3)=\sum_{1}^{\frac{N}{3}}\Chi(x)=\sum_{1}^{\frac{N}{6}}\Chi(x)+\sum_{\frac{N}{6}}^{\frac{N}{4}}\Chi(x)+\sum_{\frac{N}{4}}^{\frac{N}{3}}\Chi(x)$. Now $\lambda$ maps $X \cap  \left( \frac{N}{6},\frac{N}{4} \right)$ onto $X \cap  \left( \frac{N}{4},\frac{N}{3} \right)$, so $\sum_{\frac{N}{4}}^{\frac{N}{3}}\Chi(x)=\sum_{\frac{N}{6}}^{\frac{N}{4}}\Chi(\lambda(x))=\sum_{\frac{N}{6}}^{\frac{N}{4}}\Chi(x).$ Set $S_{1}=\sum_{1}^{\frac{N}{6}}\Chi(x),~S_{2}=\sum_{\frac{N}{6}}^{\frac{N}{4}}\Chi(x)$, so $E_{0}(3)=S_{1}+2S_{2}$. On the other hand, by Lemma \ref{lm3} we always have $h=\sum_{1}^{\frac{N}{4}}\Chi(x)=S_{1}+S_{2}$. So if $D \equiv1(mod~3)$, there are two equations 

$$\begin{array}{l} S_{1}+ S_{2} = h\\ S_{1}+ 2S_{2} = h\\ \end{array}$$ 

which imply $S_{1}=h,~S_{2}=0$, while if $D \equiv2(mod~3)$, the equations 

$$\begin{array}{l} S_{1}+ S_{2} = h\\ S_{1}+ 2S_{2} = 2h\\ \end{array}$$ 

imply $S_{1}=0,~S_{2}=h$, which proves the theorem. 
\end{proof}

For example, referring back to (\ref{eq19}) for $D=-40 \equiv2 \pmod{3},~\frac{N}{6}=6 \frac{2}{3}$, so $S_{1}=\Chi(1)+\Chi(3)=0,~S_{2}=\Chi(7)+\Chi(9)=2=h(-40)$. 

For $D=-56 \equiv1 \pmod{3},~\frac{N}{6}=9 \frac{1}{3}, ~\frac{N}{4}=14$ and $\Chi_{-56}(x)=\Chi_{8}(x)\left( \frac{x}{7} \right)$. The values are

\begin{center}
\begin{tabular}{r | r r r r c r r c}
& & & & & $\frac{N}{6}$& & & $\frac{N}{4}$ \\
 $x$ & $1$& $3$& $5$& $9$& $\uparrow$& $11$& $13$& $\uparrow$ \\ 
\hline
$\Chi(x)$ & $1$& $1$& $1$& $1$& $$& $-1$& $1$& \\ 
\end{tabular}
\end{center}

\begin{equation*}
\begin{array}{l}
S_{1}=\Chi(1)+\Chi(3)+\Chi(5)+\Chi(9)=4=h(-56) ~\text{and}~\\
S_{2}=\Chi(11)+\Chi(13)=-1+1=0. \\
\end{array} 
\end{equation*}



\begin{thebibliography}{9}

\bibitem{Berndt} 
B.C. Berndt, 
\newblock \emph{Classical theorems on quadratic residues}, 
\newblock L'Enseignment Mathematique (2) 22 (1976). 
\newblock 261-304.

\bibitem{BoSha} 
A.I. Borevich, I.R. Shafarevich,
\newblock \emph{Number Theory}, 
\newblock Academic Press, New York-London, 1966.

\bibitem{KG} 
K. Girstmair, 
\newblock \emph{A ``popular'' class number formula}, 
\newblock American Math Monthly, 101 (1994).
\newblock 997-1001.
\end{thebibliography}
\end{document}